\documentclass[a4paper]{amsart}
\usepackage[latin1]{inputenc}
\usepackage{amssymb,amsmath,amsthm}
\usepackage{amsfonts}
\usepackage{paralist}
\usepackage{ifthen}
\usepackage{url}
\usepackage[german,english]{babel}
\usepackage{rotating}

\theoremstyle{plain}
\newtheorem{theorem}{Theorem}[section]
\newtheorem{proposition}[theorem]{Proposition}
\newtheorem{lemma}[theorem]{Lemma}
\newtheorem{corollary}[theorem]{Corollary}

\theoremstyle{remark}
\newtheorem{remark}{Remark}

\newcommand{\cl}[1]{\ensuremath{\mathcal{#1}}}
\newcommand{\minor}[3]{\ensuremath{{#1}_{{#2} \leftarrow {#3}}}}
\newcommand{\vect}[1]{\ensuremath{\mathbf{#1}}} 

\begin{document}

\hyphenation{Bool-ean}

\title{Self-commuting lattice polynomial functions}
\date{\today}
\author{Miguel Couceiro}
\address[M. Couceiro]{University of Luxembourg \\
Mathematics Research Unit \\
6, rue Richard Coudenhove-Kalergi \\
L--1359 Luxembourg \\
Luxembourg}
\email{miguel.couceiro@uni.lu}
\author{Erkko Lehtonen}
\address[E. Lehtonen]{University of Luxembourg \\
Computer Science and Communications Research Unit \\
6, rue Richard Coudenhove-Kalergi \\
L--1359 Luxembourg \\
Luxembourg}
\email{erkko.lehtonen@uni.lu}

\begin{abstract}
We provide sufficient conditions for a lattice polynomial function to be self-commuting. We explicitly describe self-commuting polynomial functions over chains.
\end{abstract}

\maketitle

\section{Introduction}

Two operations $f \colon A^n \to A$ and $g \colon A^m \to A$ are said to commute, if for all $a_{ij}\in A$ ($1 \leq i \leq n$, $1 \leq j \leq m$), the following identity holds
\begin{multline*}
f \bigl( g(a_{11}, a_{12}, \dotsc, a_{1m}), g(a_{21}, a_{22}, \dotsc, a_{2m}), \dotsc, g(a_{n1}, a_{n2}, \dotsc, a_{nm}) \bigr) \\
= g \bigl( f(a_{11}, a_{21}, \dotsc, a_{n1}), f(a_{12}, a_{22}, \dotsc, a_{n2}), \dotsc, f(a_{1m}, a_{2m}, \dotsc, a_{nm}) \bigr).
\end{multline*}

For $n = m = 2$, the above condition stipulates that
\[
f \bigl( g(a_{11}, a_{12}), g(a_{21}, a_{22}) \bigr) = g \bigl( f(a_{11}, a_{21}), f(a_{12}, a_{22}) \bigr).
\]
The Eckmann-Hilton theorem~\cite{EH} asserts that if both $f$ and $g$ have an identity element and $f \perp g$, then in fact $f = g$ and $(A; f)$ is a commutative monoid on $A$.

The relevance of the notion of commutation is made apparent in works of several authors. In particular, commutation is the defining property of entropic algebras~\cite{PS,RS,Stronkowski} (an algebra is entropic if its operations commute pairwise; idempotent entropic algebras are called modes) and centralizer clones~\cite{Larose,MR,ST,TS} (the centralizer of a set $F$ of operations is the set of all operations that commute with every operation in $F$; the centralizer of $F$ is a clone).

We are interested in functions $f$ that commute with themselves. An algebra $(A; f)$ where $f$ is a binary operation that satisfies the identity
\[
f \bigl( f(a_{11}, a_{12}), f(a_{21}, a_{22}) \bigr) = f \bigl( f(a_{11}, a_{21}), f(a_{12}, a_{22}) \bigr)
\]
is called a medial groupoid~\cite{JKAU,JKRoz}. Hence, self-com\-mu\-ta\-tion generalizes the notion of mediality (see, e.g., \cite{GMMP}), and it has been investigated by several authors (see, e.g., \cite{Aczel,AD,MMT,Soublin}). In the realm of aggregation theory, self-commutation is also known as bisymmetry; for motivations and general background, see \cite{GMMP}.

In this paper, we address the question of characterizing classes of self-commuting operations. In Section~\ref{sec:preliminaries}, we recall basic notions in the universal-algebraic setting and settle the terminology used throughout the paper. Moreover,
by showing that self-commutation is preserved under several operations (e.g., permutation of variables, identification of variables and addition of dummy variables), we develop general tools for tackling the question of describing self-commuting operations.

This question is partially answered for lattice polynomial functions (in particular, for the so-called discrete Sugeno integrals, i.e., idempotent polynomial functions; see, e.g., \cite{CouMar,GMMP}) in Section~\ref{sec:SCLPF}. We start by surveying well-known results concerning normal form representations of these lattice functions which we then use to specify those polynomial functions on bounded chains which are self-commuting.
This explicit description is obtained by providing sufficient conditions for a lattice polynomial function to be self-commuting, and by showing that these conditions are also necessary in the particular case of polynomial functions over bounded chains.

In Section~\ref{sec:Concluding} we point out problems which are left unsettled, and motivate directions of future research.


\section{Preliminaries}
\label{sec:preliminaries}

In this section, we introduce some notions and terminology as well as establish some preliminary results that will be used in the sequel. For an integer $n \geq 1$, set $[n] := \{1, 2, \dotsc, n\}$. With no danger of ambiguity, we denote the tuple $(x_1, \dotsc, x_n)$ of any length by $\mathbf{x}$.


\subsection{Operations and algebras}

Let $A$ be an arbitrary nonempty set. An \emph{operation} on $A$ is a map $f \colon A^n \to A$ for some integer $n \geq 1$, called the \emph{arity} of $f$. We denote by $\cl{O}_A^{(n)}$ the set of all $n$-ary operations on $A$, and we denote by $\cl{O}_A$ the set of all finitary operations on $A$, i.e., $\cl{O}_A := \bigcup_{n \geq 1} \cl{O}_A^{(n)}$.

We assume that the reader is familiar with basic notions of universal algebra and lattice theory. In particular, the concepts of \emph{term operation} and \emph{polynomial operation} will not be defined in the current paper, and we refer the reader to~\cite{Birkhoff,BS,DP,DW,DW2009,Gratzer,Rud01} for general background on universal algebra and lattice theory.


\subsection{Simple minors}

Let $f \in \cl{O}_A^{(n)}$, $g \in \cl{O}_A^{(m)}$. We say that $f$ is obtained from $g$ by \emph{simple variable substitution,} or $f$ is a \emph{simple minor} of $g$, if there is a mapping $\sigma \colon [m] \to [n]$ such that
\[
f(x_1, \dotsc, x_n) = g(x_{\sigma(1)}, x_{\sigma(2)}, \dotsc, x_{\sigma(m)}).
\]
If $\sigma$ is not injective, then we speak of \emph{identification of variables.} If $\sigma$ is not surjective, then we speak of \emph{addition of inessential variables.} If $\sigma$ is bijective, then we speak of \emph{permutation of variables.} For distinct indices $i, j \in [n]$, the function $\minor{f}{i}{j} \colon A^n \to A$ obtained from $f$ by the simple variable substitution
\[
\minor{f}{i}{j}(x_1, \dotsc, x_n) := f(x_1, \dotsc, x_{i-1}, x_j, x_{i+1}, \dotsc, x_n)
\]
is called a \emph{variable identification minor} of $f$, obtained by identifying $x_i$ with $x_j$.

For studies of classes of operations that are closed under taking simple minors, see, e.g.,~\cite{CF2005,Pippenger}.


\subsection{Self-commutation}

Let $f \colon A^n \to A$ and $g \colon A^m \to A$ be operations on $A$. We say that $f$ \emph{commutes} with $g$, denoted $f \perp g$, if for all $a_{ij}$ ($i \in [n]$, $j \in [m]$), it holds that
\begin{multline*}
f \bigl( g(a_{11}, a_{12}, \dotsc, a_{1m}), g(a_{21}, a_{22}, \dotsc, a_{2m}), \dotsc, g(a_{n1}, a_{n2}, \dotsc, a_{nm}) \bigr) \\
= g \bigl( f(a_{11}, a_{21}, \dotsc, a_{n1}), f(a_{12}, a_{22}, \dotsc, a_{n2}), \dotsc, f(a_{1m}, a_{2m}, \dotsc, a_{nm}) \bigr).
\end{multline*}
If $f \perp f$, then we say that $f$ is \emph{self-commuting.}

\begin{lemma}
\label{lem:scminors}
Let $f \in \cl{O}_A^{(n)}$, $g \in \cl{O}_A^{(m)}$, and let $\sigma \colon [n] \to [\nu]$ and $\tau \colon [m] \to [\mu]$ be arbitrary mappings. Let $f_\sigma \in \cl{O}_A^{(\nu)}$ and $g_\tau \in \cl{O}_A^{(\mu)}$ be the operations defined by
\begin{align*}
f_\sigma(x_1, \dots, x_{\nu}) &= f(x_{\sigma(1)}, \dots, x_{\sigma(n)}), \\
g_\tau(x_1, \dots, x_{\mu}) &= g(x_{\tau(1)}, \dots, x_{\tau(m)}).
\end{align*}
If $f \perp g$, then $f_\sigma \perp g_\tau$.
\end{lemma}

\begin{proof}
By the definition of $f_\sigma$ and $g_\tau$,
\begin{align*}
& f_\sigma \bigl( g_\tau(a_{11}, a_{12}, \dotsc, a_{1m}), g_\tau(a_{21}, a_{22}, \dotsc, a_{2m}), \dotsc, g_\tau(a_{n1}, a_{n2}, \dotsc, a_{nm}) \bigr) = \\
& f \bigl( g(a_{\sigma(1) \tau(1)}, a_{\sigma(1) \tau(2)}, \dotsc, a_{\sigma(1) \tau(m)}), g(a_{\sigma(2) \tau(1)}, a_{\sigma(2) \tau(2)}, \dotsc, a_{\sigma(2) \tau(m)}), \dotsc, \\ & \qquad\qquad\qquad\qquad\qquad\qquad\qquad\qquad\quad g(a_{\sigma(n) \tau(1)}, a_{\sigma(n) \tau(2)}, \dotsc, a_{\sigma(n) \tau(m)}) \bigr) = \\
& g \bigl( f(a_{\sigma(1) \tau(1)}, a_{\sigma(2) \tau(1)}, \dotsc, a_{\sigma(m) \tau(1)}), f(a_{\sigma(1) \tau(2)}, a_{\sigma(2) \tau(2)}, \dotsc, a_{\sigma(m) \tau(2)}), \dotsc, \\ & \qquad\qquad\qquad\qquad\qquad\qquad\qquad\qquad\quad f(a_{\sigma(1) \tau(n)}, a_{\sigma(2) \tau(n)}, \dotsc, a_{\sigma(n) \tau(m)}) \bigr) = \\
& g_\tau \bigl( f_\sigma(a_{11}, a_{21}, \dotsc, a_{m1}), f_\sigma(a_{12}, a_{22}, \dotsc, a_{m2}), \dotsc, f_\sigma(a_{1n}, a_{2n}, \dotsc, a_{nm}) \bigr),
\end{align*}
where the second equality holds by the assumption that $f \perp g$.
\end{proof}

\begin{corollary}
\label{cor:smsc}
If $f \in \cl{O}_A$ is self-commuting, then every simple minor of $f$ is self-commuting.
\end{corollary}

In the particular case when $A$ is finite, Corollary~\ref{cor:smsc} translates into saying that the class of self-commuting operations on $A$ is definable by functional equations (see \cite{CF2007}).

The set of self-commuting operations is also closed under special type of substitutions of constants for variables, as described by the following lemma.
Let $f \colon A^n \to A$ and $c \in A$. For $i \in [n]$, we define $f_c^i \colon A^{n-1} \to A$ to be the operation
\[
f_c^i(a_1, \dots, a_{n-1}) = f(a_1, \dots, a_{i-1}, c, a_i, \dots, a_{n-1}).
\]

\begin{lemma}
Assume that $f \colon A^n \to A$ preserves $c \in A$, i.e., $f(c, \dotsc, c) = c$. If $f$ is self-commuting, then for every $i \in [n]$, $f_c^i$ is self-commuting.
\end{lemma}

\begin{proof}
We will show that the claim holds for $i = 1$. It the follows from Lemma~\ref{lem:scminors}, by considering suitable permutations of variables, that the claim holds for all $i \in [n]$.
By the definition of $f_c^1$ and by the assumption that $f(c, \dots, c) = c$, we have
\begin{multline*}
f_c^1 \bigl( f_c^1(a_{11}, \dotsc, a_{1,n-1}), \dotsc, f_c^1(a_{n-1,1}, \dotsc, a_{n-1,n-1}) \bigr) = \\
f \bigl( f(c, c, \dotsc, c), f(c, a_{11}, \dotsc, a_{1,n-1}), \dotsc, f(c, a_{n-1,1}, \dotsc, a_{n-1,n-1}) \bigr) = \\
f \bigl( f(c, c, \dotsc, c), f(c, a_{11}, \dotsc, a_{n-1,1}), \dotsc, f(c, a_{1,n-1}, \dotsc, a_{n-1,n-1}) \bigr) = \\
f_c^1 \bigl( f_c^1(a_{11}, \dotsc, a_{n-1,1}), \dotsc, f_c^1(a_{1,n-1}, \dotsc, a_{n-1,n-1}) \bigr),
\end{multline*}
where the second equality holds by the assumption that $f$ is self-commuting.
\end{proof}


\section{Self-commuting lattice polynomial functions}
\label{sec:SCLPF}

Let $(L; \wedge, \vee)$ be a lattice. With no danger of ambiguity, we denote lattices by their universes. In this section we study the self-commutation property on  lattice polynomial functions, i.e., mappings $f\colon L^n\to L$ which can be obtained as compositions of the lattice operations and applied to variables (projections) and constants. 
As shown by Goodstein \cite{Goodstein}, lattice polynomial functions have neat normal form representations in the case when $L$ is a bounded distributive lattice. Thus, in what follows we assume that $L$ is a bounded distributive lattice with least and greatest elements $0$ and $1$, respectively.

We recall the necessary representation results concerning the representation of lattice polynomials as well as introduce some related concepts and terminology in Subsection~\ref{representation}. Then, we consider the property of self-commutation on these functions. We start by providing sufficient conditions for a lattice polynomial function to be self-commuting, which we then use to obtain explicit descriptions of those polynomial functions on chains which satisfy this self-commutation property.

  
\subsection{Preliminary results: representations of lattice polynomials}
\label{representation}

An $n$-\emph{ary} (\emph{lattice}) \emph{polynomial function} from $L^n$ to $L$ is defined recursively as follows:
\begin{enumerate}
\item[(i)] For each $i\in [n]$ and each $c\in L$, the projection $\vect{x}\mapsto x_i$ and the constant function $\vect{x}\mapsto c$ are
polynomial functions from $L^n$ to $L$.

\item[(ii)] If $f$ and $g$ are polynomial functions from $L^n$ to $L$, then $f\vee g$ and $f\wedge g$ are polynomial functions from $L^n$ to
$L$.

\item[(iii)] Any  polynomial function from $L^n$ to $L$ is obtained by finitely many applications of the rules (i) and (ii).
\end{enumerate}

If rule (i) is only applied for projections, then the resulting polynomial functions are called (\emph{lattice}) \emph{term functions}~\cite{BS,Gratzer,DW}. Idempotent polynomial functions are also referred to as (\emph{discrete}) \emph{Sugeno integrals}~\cite{CouMar,GMMP}. In the case of bounded distributive lattices, Goodstein \cite{Goodstein} showed that polynomial functions are exactly those which allow
representations in disjunctive normal form (see Proposition~\ref{DNF} below, first appearing in \cite[Lemma 2.2]{Goodstein}; see
also Rudeanu~\cite[Chapter~3,\,\S{3}]{Rud01} for a later reference).

\begin{proposition}\label{DNF}
Let $L$ be a bounded distributive lattice. A function $f \colon L^n \to L$ is a polynomial function if and only if there exist $a_I\in L$, $I\subseteq [n],$ such that, for every $\vect{x}\in L^n$,
$$
f(\vect{x})=\bigvee_{I\subseteq [n]}(a_I\wedge \bigwedge_{i\in I} x_i).
$$
\end{proposition}

The expression given in Proposition~\ref{DNF} is usually referred to as the \emph{disjunctive normal form} (DNF)
representation of the polynomial function $f$. In order to simplify notation, if $I$ is a singleton or a two-element set, then we write $a_i$ and $a_{ij}$ for $a_{\{i\}}$ and $a_{\{i,j\}}$, respectively.

The following corollaries belong to the folklore of lattice theory and are immediate consequences of Theorems D and E in \cite{Goodstein}.

\begin{corollary}\label{Ext1}
Every polynomial function is completely determined by its restriction to $\{0,1\}^n$.
\end{corollary}

\begin{corollary}\label{Ext2}
A function $g\colon \{0,1\}^n\rightarrow L$ can be extended to a polynomial function $f\colon L^{n}\rightarrow L$ if and only if it is
nondecreasing. In this case, the extension is unique.
\end{corollary}

It is easy to see that the DNF representations of a polynomial function $f\colon L^{n}\rightarrow L$ are not necessarily unique. For instance, 
in Proposition~\ref{DNF}, if for some $I\subseteq [n]$ we have $a_I=\bigvee_{J\subsetneq I}a_J$, then for every $\vect{x}\in L^n$,
\[
f(\vect{x})=\bigvee_{I \neq J \subseteq [n]}(a_J\wedge \bigwedge_{i\in J} x_i).
\]
We refer to the term $a_I \bigwedge_{i \in I} x_i$ as the \emph{$I$-th term} of $f$, and we say that $\lvert I \rvert$ is its \emph{size.} We say that the $I$-th term $a_I \bigwedge_{i \in I} x_i$ is \emph{essential} if $a_I > \bigvee_{J \subsetneq I} a_J$; otherwise, we say that it is \emph{inessential}.
(For a discussion on the uniqueness of DNF representations of lattice polynomial functions see \cite{CouMar}.)

However, using Corollaries \ref{Ext1} and \ref{Ext2}, one can easily set canonical ways of constructing these normal form representations of polynomial functions.

Let $2^{[n]}$ denote the set of all subsets of $[n]$. For $I \subseteq [n]$, let $\vect{e}_I$ be the
\emph{characteristic vector} of $I$, i.e., the $n$-tuple in $L^n$ whose $i$-th component is $1$ if $i \in I$, and 0 otherwise.
Note that the mapping $\alpha \colon 2^{[n]} \to \{0, 1\}^n$ given by $\alpha(I) = \vect{e}_I$, for every $I \in 2^{[n]}$, is an order-isomorphism. 

\begin{proposition}[{Goodstein \cite{Goodstein}}]\label{prop:DNF(f)}
Let $L$ be a bounded distributive lattice. A function $f \colon L^n \to L$ is a polynomial function if and only if for every $\vect{x} \in L^n$,
\[
f(\vect{x}) = \bigvee_{I \subseteq [n]} \bigl( f(\vect{e}_I) \wedge \bigwedge_{i \in I} x_i \bigr).
\]
\end{proposition}

It is noteworthy that Proposition~\ref{prop:DNF(f)} leads to the following characterization of the essential arguments of polynomial functions in terms of necessary and sufficient conditions \cite{CL-ISMVL}.

\begin{proposition}
\label{prop:essential}
Let $L$ be a bounded distributive lattice and let $f \colon L^n \to L$ be a polynomial function.  Then for each $j \in [n]$, $x_j$ is essential in $f$ if and only if there exists a set $J \subseteq [n] \setminus \{j\}$ such that $ f(\vect{e}_J)< f(\vect{e}_{J \cup \{j\}})$.
\end{proposition}

\begin{remark}
The assumption that the lattice $L$ is bounded is not very crucial. Let $L'$ be the lattice obtained from $L$ by adjoining new top and bottom elements $\top$ and $\bot$, if necessary. Then, if $f$ is a polynomial function over $L$ induced by a polynomial $p$, then $p$ induces a polynomial function $f'$ on $L'$, and it holds that the restriction of $f'$ to $L$ coincides with $f$. Similarly, if $L'$ is a distributive lattice and $f'$ is a polynomial function on $L'$ represented by the DNF
\[
\bigvee_{I \subseteq [n]} (a_I \wedge \bigwedge_{i \in I} x_i),
\]
then by omitting each term $a_I \wedge \bigwedge_{i \in I} x_i$ where $a_I = \bot$ and replacing each term $a_I \wedge \bigwedge_{i \in I} x_i$ where $a_I = \top$ by $\bigwedge_{i \in I} x_i$, we obtain an equivalent polynomial representation for $f'$. Unless $f'$ is the constant function that takes value $\top$ or $\bot$ and this element is not in $L$, the function $f$ on $L$ induced by this new polynomial coincides with the restriction of $f'$ to $L$.
\end{remark}


\subsection{Self-commuting polynomial functions on chains}

In this subsection we provide explicit descriptions of self-commuting polynomial functions on chains.

A lattice polynomial function $f \colon L^n \to L$ is said to be a \emph{weighted disjunction} if it is of the form
\begin{equation}
f(x_1, x_2, \dotsc, x_n) = a_\emptyset \vee \bigvee_{i \in [n]} a_i x_i
\label{eq:bisymmetric1}
\end{equation}
for some elements $a_\emptyset$, $a_i$ ($i \in [n]$) of $L$. We say that $f$ has \emph{chain form} if
\begin{equation}
f(x_1, x_2, \dotsc, x_n) = a_\emptyset \vee \bigvee_{i \in [n]} a_i x_i \vee \bigvee_{1 \leq \ell \leq r} a_{S_\ell} \bigwedge_{i \in S_\ell} x_i,
\label{eq:bisymmetric}
\end{equation}
for a chain of subsets $S_1 \subseteq S_2 \subseteq \dotsb \subseteq S_r \subseteq [n]$, $r \geq 1$, $\lvert S_1 \rvert \geq 2$, and some elements $a_\emptyset$, $a_i$ ($i \in [n]$), $a_{S_\ell}$ ($1 \leq \ell \leq r$) of $L$ such that $a_I \leq a_J$ whenever $I \subseteq J$, and for all $i \notin S_1$, there is a $j \in S_1$ such that $a_i \leq a_j$.

\begin{theorem}
\label{main}
Let $L$ be a bounded chain. A polynomial function $f \colon L^n \to L$ is self-commuting if and only if it is a weighted disjunction or it has chain form.
\end{theorem}

Theorem \ref{main} will be a consequence of the following two results.
We start with a lemma that provides sufficient conditions for a polynomial to be self-commuting in the general case of bounded distributive lattices. 

\begin{lemma}
\label{lemma:sufficiency}
Let $L$ be a distributive lattice. Assume that a function $f \colon L^n \to L$ is a weighted disjunction or has chain form.
Then $f$ is self-commuting.
\end{lemma}

\begin{proof}
Assume first that $f$ is a weighted disjunction. We have that
\begin{multline*}
f \bigl( f(x_{11}, x_{12}, \dots, x_{1n}), \dots, f(x_{n1}, x_{n2}, \dots, x_{nn}) \bigr) \\
 = a_\emptyset \vee \bigvee_{i \in [n]} a_i (a_\emptyset \vee \bigvee_{j \in [n]} a_j x_{ij} )
 = a_\emptyset \vee \bigvee_{i \in [n]} \bigvee_{j \in [n]} a_i a_j x_{ij} \\
 = a_\emptyset \vee \bigvee_{j \in [n]} \bigvee_{i \in [n]} a_j a_i x_{ij}
 = a_\emptyset \vee \bigvee_{j \in [n]} a_j (a_\emptyset \vee \bigvee_{i \in [n]} a_i x_{ij} ) \\
 = f \bigl( f(x_{11}, x_{21}, \dots, x_{n1}), \dots, f(x_{1n}, x_{2n}, \dots, x_{nn}) \bigr).
\end{multline*}
Thus, $f$ is self-commuting.

Assume then that $f$ has chain form. The assumption that for every $i \notin S_1$ there is a $j \in S_1$ such that $a_i \leq a_j$ implies that $a_i \leq a_{S_\ell}$ (and hence $a_i a_{S_\ell} = a_i$) for all $i \in [n]$ and for all $\ell \in [r]$. Using this observation and the distributive laws we get
\[
\begin{split}
& \!\!\!\!\!\!\!\!\!\!\!\!\!\!\!\!\!\!\!\!\!\! f \bigl( f(x_{11}, x_{12}, \dotsc, x_{1n}), f(x_{21}, x_{22}, \dotsc, x_{2n}), \dotsc, f(x_{n1}, x_{n2}, \dotsc, x_{nn}) \bigr) \\
= a_\emptyset & \vee \bigvee_{i \in [n]} a_i \Bigl[ a_\emptyset \vee \bigvee_{j \in [n]} a_j x_{ij} \vee \bigvee_{1 \leq \ell \leq r} a_{S_\ell} \bigwedge_{j \in S_\ell} x_{ij} \Bigr] \\ & \vee \bigvee_{1 \leq t \leq r} a_{S_t} \bigwedge_{i \in S_t} \Bigl[ a_\emptyset \vee \bigvee_{j \in [n]} a_j x_{ij} \vee \bigvee_{1 \leq \ell \leq r} a_{S_\ell} \bigwedge_{j \in S_\ell} x_{ij} \Bigr] \\
= a_\emptyset & \vee \underbrace{\bigvee_{i \in [n]} \bigvee_{j \in [n]} a_i a_j x_{ij}}_{\text{(I)}} \vee \underbrace{\bigvee_{i \in [n]} \bigvee_{1 \leq \ell \leq r} a_i \bigwedge_{j \in S_\ell} x_{ij}}_{\text{(II)}} \\ & \vee \underbrace{\bigvee_{1 \leq t \leq r} \bigwedge_{i \in S_t} \Bigl[ a_\emptyset \vee \bigvee_{j \in [n]} a_j x_{ij} \vee \bigvee_{1 \leq \ell \leq r} a_{S_t} a_{S_\ell} \bigwedge_{j \in S_\ell} x_{ij} \Bigr]}_{\text{(III)}}.
\end{split}
\]
Every term in (II) is absorbed by a term in (I): for every $i \in [n]$, there is a $k \in S_1$ such that $a_i \leq a_k$, and hence for any $\ell \in [r]$, the term $a_i \bigwedge_{j \in S_\ell} x_{ij} = a_i a_k x_{ik} \bigwedge_{j \in S_\ell \setminus \{k\}} x_{ij}$ in (II) is absorbed by the term $a_i a_k x_{ik}$ in (I).

In (III), for a fixed $t$, if $\ell > t$, then the term $a_{S_t} a_{S_\ell} \bigwedge_{j \in S_\ell} x_{ij} = a_{S_t} \bigwedge_{j \in S_\ell} x_{ij}$ is absorbed by $a_{S_t} \bigwedge_{j \in S_t} x_{ij} = a_{S_t} a_{S_t} \bigwedge_{j \in S_t} x_{ij}$, and hence (III) simplifies to
\begin{equation}
\label{eqIV}
\bigvee_{1 \leq t \leq r} \underbrace{\bigwedge_{i \in S_t} \Bigl[ a_\emptyset \vee \bigvee_{j \in [n]} a_j x_{ij} \vee \bigvee_{1 \leq \ell \leq t} a_{S_\ell} \bigwedge_{j \in S_\ell} x_{ij} \Bigr]}_{\text{(IV)}}.
\end{equation}

For a fixed $t$, (IV) expands to the disjunction of all possible conjunctions $\bigwedge_{i \in S_t} \phi_i$ of $\lvert S_t \rvert$ terms, where each $\phi_i$ is one of $a_\emptyset$, $a_j x_{ij}$ for some $j \in [n]$, or $a_{S_\ell} \bigwedge_{j \in S_\ell} x_{ij}$ for some $1 \leq \ell \leq t$. If $\phi_i = a_\emptyset$ for some $i \in S_t$, then the conjunction is absorbed by $a_\emptyset$. If $\phi_i = a_i x_{ii}$ for some $i \in S_t$, then the conjunction is absorbed by the term $a_i a_i x_{ii} = a_i x_{ii}$ in (I).

Consider then such a conjunction $\bigwedge_{i \in S_t} \phi_i$ where for all $i \in S_t$, $\phi_i$ is not equal to $a_\emptyset$ nor to $a_i x_{ii}$, but for some $i \in S_t$, $\phi_i = a_j x_{ij}$ for some $j \neq i$. By our assumption, there is a $k \in S_1$ such that $a_j \leq a_k$ and hence $a_j = a_j a_k$. We have that $\phi_k$ equals either $a_\ell x_{k \ell}$ for some $\ell \neq k$ or $a_{S_\ell} \bigwedge_{m \in S_\ell} x_{km}$ for some $1 \leq \ell \leq t$. In the former case, $\phi_i \phi_k = a_j a_k x_{ij} a_\ell x_{k \ell}$, and hence the conjunction $\bigwedge_{i \in S_t} \phi_i$ is absorbed by the term $a_k a_\ell x_{k \ell}$ in (I). In the latter case, $\phi_i \phi_k = a_j a_k x_{ij} a_{S_\ell} \bigwedge_{m \in S_\ell} x_{km}$, and hence the conjunction $\bigwedge_{i \in S_t} \phi_i$ is absorbed by the term $a_k a_k x_{kk} = a_k x_{kk}$ in (I).

The remaining conjunctions that arise from the expansion of (IV) are of the form
\[
\bigwedge_{i \in S_t} a_{S_{\ell_i}} \bigwedge_{j \in S_{\ell_i}} x_{ij}
\]
where $1 \leq \ell_i \leq t$ ($i \in S_t$). Let $\ell' = \min_{i \in S_t} \ell_i$. If $\ell' < t$, then this conjunction is absorbed by $a_{S_{\ell'}} \bigwedge_{i \in S_{\ell'}} \bigwedge_{j \in S_{\ell'}} x_{ij}$, which arises from the expansion of
\[
\bigwedge_{i \in S_{\ell'}} \Bigl[ a_\emptyset \vee \bigvee_{j \in [n]} a_j x_{ij} \vee \bigvee_{1 \leq \ell \leq \ell'} a_{S_\ell} \bigwedge_{j \in S_\ell} x_{ij} \Bigr]
\]
in \eqref{eqIV}. Thus, the only remaining conjunction that arises from the expansion of (IV) is $a_{S_t} \bigwedge_{i \in S_t} \bigwedge_{j \in S_t} x_{ij}$.

Thus, we have that
\begin{multline}
f \bigl( f(x_{11}, x_{12}, \dotsc, x_{1n}), f(x_{21}, x_{22}, \dotsc, x_{2n}), \dotsc, f(x_{n1}, x_{n2}, \dotsc, x_{nn}) \bigr) = \\
a_\emptyset \vee \bigvee_{i \in [n]} \bigvee_{j \in [n]} a_i a_j x_{ij} \vee \bigvee_{1 \leq \ell \leq r} a_{S_\ell} \bigwedge_{i \in S_\ell} \bigwedge_{j \in S_\ell} x_{ij}.
\label{rh1}
\end{multline}
In a similar way, we can deduce that
\begin{multline}
f \bigl( f(x_{11}, x_{21}, \dotsc, x_{n1}), f(x_{12}, x_{22}, \dotsc, x_{n2}), \dotsc, f(x_{1n}, x_{2n}, \dotsc, x_{nn}) \bigr) = \\
a_\emptyset \vee \bigvee_{j \in [n]} \bigvee_{i \in [n]} a_i a_j x_{ij} \vee \bigvee_{1 \leq \ell \leq r} a_{S_\ell} \bigwedge_{j \in S_\ell} \bigwedge_{i \in S_\ell} x_{ij}.
\label{rh2}
\end{multline}
The right hand sides of \eqref{rh1} and \eqref{rh2} are clearly equal, and we conclude that $f$ is self-commuting.
\end{proof}

The necessity of the conditions in Theorem~\ref{main} follows from our next lemma.

\begin{lemma}
\label{lemma:necessity}
Let $L$ be a bounded chain. If a polynomial function $f \colon L^n \to L$ is self-commuting, then it is a weighted disjunction or it has chain form.
\end{lemma}

\begin{proof}
The statement clearly holds for $n = 1$ and $n = 2$, since every unary or binary polynomial function is a weighted disjunction or has chain form.

Suppose $n = 3$. Then
\begin{equation}
f = a_\emptyset \vee a_1 x_1 \vee a_2 x_2 \vee a_3 x_3 \vee a_{12} x_1 x_2 \vee a_{13} x_1 x_3 \vee x_{23} x_2 x_3 \vee a_{123} x_1 x_2 x_3,
\label{eq:f3}
\end{equation}
where $a_I \leq a_J$ whenever $I \subseteq J$.
If for all $i, j \in \{1,2,3\}$, $a_i \vee a_j = a_{ij}$, then each term $a_{ij} x_i x_j$ in \eqref{eq:f3} equals $(a_i \vee a_j) x_i x_j = a_i x_i x_j \vee a_j x_i x_j$ and gets absorbed by $a_i x_i$ and $a_j x_j$, and hence $f$ has the desired form \eqref{eq:bisymmetric1} or \eqref{eq:bisymmetric}. Otherwise, there exist $i, j$ such that $a_i \vee a_j < a_{ij}$; without loss of generality, assume that $a_1 \vee a_2 < a_{12}$.

We have that
\begin{align}
f \bigl( f( 1, 1, 0), f( 0, 1, 1), f( 0, 0, 0) \bigr) &=
a_1 \vee a_2 \vee a_{12} a_{23}, \label{eq:01matrix} \\
f \bigl( f( 1, 0, 0), f( 1, 1, 0), f( 0, 1, 0) \bigr) &=
a_1 \vee a_2, \nonumber
\end{align}
and since $f$ is self-commuting, we have $a_1 \vee a_2 \vee a_{12} a_{23} = a_1 \vee a_2$. This equality translates into $ a_{12} a_{23} \leq a_1 \vee a_2$. In a similar way, after suitably permuting the rows and columns of the $3 \times 3$ matrix used in \eqref{eq:01matrix}, we can deduce that
\begin{equation}
\label{eq:aij}
a_{ij} a_{jk} \leq a_i \vee a_j \leq a_{ij}
\end{equation}
for $\{i, j, k\} = \{1, 2, 3\}$.

Since $L$ is a chain, we have for some choice of $\{\alpha, \beta, \gamma\} = \{1, 2, 3\}$ that $a_{\alpha \beta} \leq a_{\beta \gamma} \leq a_{\alpha \gamma}$. Inequalities \eqref{eq:aij} then imply
\[
a_\alpha \vee a_\beta = a_{\alpha \beta}
\quad
\text{and}
\quad
a_\beta \vee a_\gamma = a_{\alpha \gamma},
\]
i.e., the terms associated with sets $\{\alpha, \beta\}$ and $\{\alpha, \gamma\}$ are inessential. Thus, $f$ has at most one essential term of size $2$. If $f$ has no essential term of size $2$, then either it is a weighted disjunction or it has chain form with $S_1 = \{1, 2, 3\}$. Otherwise $f$ has precisely one essential term of size $2$, say, associated with $S_1 = \{1, 2\}$. Then $a_{12} > a_1 \vee a_2$ and
\[
a_3 \leq a_{13} = a_{13} a_{12} \leq a_1 \vee a_2.
\]
Since $L$ is a chain, $a_3 \leq a_1$ or $a_3 \leq a_2$, and we conclude that $f$ has chain form.

We proceed by induction on $n$. Assume that the claim holds for $n < \ell$ for some $\ell \leq 4$. We show that it holds for $n = \ell$.

Let $f = a_\emptyset \vee \bigvee_{I \subseteq [\ell]} a_I \bigwedge_{i \in I} x_i$ be self-commuting, and assume that $a_I \leq a_J$ whenever $I \subseteq J$. If $f$ has no essential terms of size at least $2$, then $f$ is a weighted disjunction. Thus, we suppose that $f$ has an essential term of size at least $2$. First we show that the essential terms of $f$ of size at least $2$ are associated with a chain $S_1 \subseteq S_2 \subseteq \dots \subseteq S_q$. For a contradiction, suppose that there are $I, J \subseteq [k]$ such that $\lvert I \rvert \geq 2$, $\lvert J \rvert \geq 2$, $I \parallel J$ and the $I$-th and the $J$-th terms of $f$ are essential. Fix such $I$ and $J$ so that $\lvert I \cap J \rvert$ is the largest possible, $\lvert I \rvert \leq \lvert J \rvert$ and $\lvert J \rvert$ is the largest among such pairs. We will consider several cases.

Case 1: $\lvert I \cap J \rvert \geq 2$. Take distinct $i, j \in I \cap J$, and consider $\minor{f}{i}{j}$. This function is a polynomial function having essential terms $b_{I'} \bigwedge_{i \in I'} x_i$ and $b_{J'} \bigwedge_{i \in J'} x_i$ where $I' = I \setminus \{i\}$, $J' = J \setminus \{i\}$. Since $I' \parallel J'$, the induction hypothesis implies that $\minor{f}{i}{j}$ is not self-commuting, which contradicts Corollary~\ref{cor:smsc} which asserts that self-commutation is preserved by taking simple minors.

Case 2: $\lvert I \cap J \rvert \leq 1$ and $\lvert J \rvert \geq 3$. Take distinct $i, j \in J \setminus I$, and considert $\minor{f}{i}{j}$. As in Case~1, we derive a contradiction, because this function has essential terms $b_{I'} \bigwedge_{i \in I'} x_i$ and $b_{J'} \bigwedge_{i \in J'} x_i$ where $I' = I$, $J' = J \setminus \{i\}$ and $I' \parallel J'$.

Case 3: $\lvert I \cap J \rvert = 0$ and $\lvert J \rvert = 2$. Take $i \in I$, $j \in J$, and consider $\minor{f}{i}{j}$. Again, we derive a contradiction, because this function has essential terms $b_{I'} \bigwedge_{i \in I'} x_i$ and $b_{J'} \bigwedge_{i \in J'} x_i$ where $I' = (I \setminus \{i\}) \cup \{j\}$, $J' = J$ and $I' \parallel J'$.

Case 4: $\lvert I \cap J \rvert = 1$, $\lvert J \rvert = 2$ and $\ell \geq 5$. Take distinct $i, j \in [\ell] \setminus (I \cup J)$, and consider $\minor{f}{i}{j}$. Again, we derive a contradiction, because this function has essential terms $b_{I'} \bigwedge_{i \in I'} x_i$ and $b_{J'} \bigwedge_{i \in J'} x_i$ where $I' = I$, $J' = J$ and $I' \parallel J'$.

Case 5: $\lvert I \cap J \rvert = 1$, $\lvert J \rvert = 2$ and $\ell = 4$. We have that
\begin{align}
f \bigl( f( 1, 1, 0, 0), f( 0, 1, 1, 0), f( 0, 0, 0, 0), f( 0, 0, 0, 0) \bigr) &=
a_1 \vee a_2 \vee a_{12} a_{23}, \label{eq:case5} \\
f \bigl( f( 1, 0, 0, 0), f( 1, 1, 0, 0), f( 0, 1, 0, 0), f( 0, 0, 0, 0) \bigr) &=
a_1 \vee a_2, \nonumber
\end{align}
and since $f$ is self-commuting, we have $a_1 \vee a_2 \vee a_{12} a_{23} = a_1 \vee a_2$. This equality translates into $a_{12} a_{23} \leq a_1 \vee a_2$. In a similar way, after suitably permuting the rows and columns of the $4 \times 4$ matrix used in \eqref{eq:case5}, we can deduce that
\begin{equation}
\label{eq:case5b}
a_{ij} a_{jk} \leq a_i \vee a_j \leq a_{ij}
\end{equation}
for distinct $i, j, k \in \{1, 2, 3, 4\}$. Assume, without loss of generality, that $I = \{1, 2\}$, $J = \{2, 3\}$. Since $L$ is a chain, we either have $a_{12} \leq a_{23}$ or $a_{23} < a_{12}$. In the former case, by \eqref{eq:case5b}, we have
\[
a_{12} = a_{12} a_{23} \leq a_1 \vee a_2 \leq a_{12},
\]
which implies that $a_{12} = a_1 \vee a_2$, which contradicts the assumption that the $I$-th term of $f$ is essential. In the latter case, we have
\[
a_{23} = a_{23} a_{12} \leq a_2 \vee a_3 \leq a_{23},
\]
which implies that $a_{23} = a_2 \vee a_3$, which contradicts the assumption that the $J$-th term of $f$ is essential.

Thus, the essential terms of $f$ of size at least $2$ are associated with a chain $S_1 \subseteq S_2 \subseteq \dots \subseteq S_q$. To complete the proof, we need to show that for every $i \notin S_1$, there is a $j \in S_1$ such that $a_i \leq a_j$. For a contradiction, suppose that there is an $i \notin S_1$ such that $a_i > a_j$ for every $j \in S_1$. We consider several cases.

Case 1: $\lvert S_1 \rvert \geq 3$. Take distinct $k, m \in S_1$, and consider $\minor{f}{k}{m}$. The essential terms of $\minor{f}{k}{m}$ of size at least $2$ are associated with a chain $S'_1 \subseteq S'_2 \subseteq \dots \subseteq S'_q$, where $S'_i := S_i \setminus \{k\}$ for $1 \leq i \leq q$, and the $m$-th term of $f$ is $(a_k \vee a_m) x_m$. Since $a_i > a_j$ for every $j \in S_1$, the induction hypothesis implies that $\minor{f}{k}{m}$ is not self-commuting. This contradicts Corollery~\ref{cor:smsc} which asserts that self-commutation is preserved by taking simple minors.

Case 2: $\lvert S_1 \rvert = 2$. Then there is a $t \in [\ell] \setminus (S_1 \cup \{i\})$. Consider $\minor{f}{t}{i}$. The essential terms of $\minor{f}{t}{i}$ of size at least $2$ are associated with a chain whose least element is $S_1$, and the $i$-th term of this function is $(a_i \vee a_t) x_i$. Since for every $j \in S_1$, $a_i > a_j$, we also have $a_i \vee a_t > a_j$, and, as above, we have reached the desired contradiction.
\end{proof}

\begin{proof}[Proof of Theorem~\ref{main}]
Lemma~\ref{lemma:sufficiency}, when restricted to chains, shows that the condition is sufficient. Necessity follows from Lemma~\ref{lemma:necessity}.
\end{proof}


\section{Concluding remarks and future work}
\label{sec:Concluding}

We have obtained an explicit form of self-commuting polynomial functions on chains (in fact, unique up to addition of inessential terms). As Lemma~\ref{lemma:sufficiency} asserts, our condition is sufficient in the general case of polynomial functions over distributive lattices. However, we do not know whether it is also a necessary condition in the general case. This constitutes a topic of ongoing research.

Another problem which was not addressed concerns commutation. As mentioned, self-commutation appears within the scope of aggregation function theory under the name of bisymmetry. In this context, functions are often regarded as mappings $f \colon \bigcup_{n \geq 1} A^n \to A$. In this framework, bisymmetry is naturally generalized to what is referred to as strong bisymmetry. Denoting by $f_n$ the restriction of $f$ to $A^n$, the map $f$ is said to be \emph{strongly bisymmetric} if for any $n, m \geq 1$, we have $f_n \perp f_m$. This generalization is both natural and useful from the application point of view. To illustrate this, suppose one is given data in tabular form, say an $n \times m$ matrix, to be meaningfully fused into a single representative value. One could first aggregate the data by rows and then aggregate the resulting column; or one could first aggregate the columns and then the resulting row. What is expressed by the property of strong bisymmetry is that the final outcome is the same under both procedures. Extending the notion of polynomial functions to such families, we are thus left with the problem of describing those families of polynomial functions which are strongly bisymmetric.


\section*{Acknowledgments}

We would like to thank Jean-Luc Marichal for introducing us to the topic and for helpful discussions.



\begin{thebibliography}{99}  

\bibitem{Aczel}
    \textsc{J. Aczél,}
    On mean values,
    \textit{Bull.\ Amer.\ Math.\ Soc.}\ \textbf{54} (1948) 392--400.


\bibitem{AD}
    \textsc{J. Aczél,}  \textsc{J. Dhombres,}
    \textit{Functional Equations in Several Variables,}
    Encyclopedia of Mathematics and Its Applications, vol.\ 31,
    Cambridge University Press, Cambridge, 1989.


\bibitem{Birkhoff}
    \textsc{G. Birkhoff,}
    \textit{Lattice Theory,} 3rd edition,
    Coll.\ Publ., XXV,
    American Mathematical Society, 1967.


\bibitem{BS}
    \textsc{S. Burris,}  \textsc{H.~P. Sankappanavar,}
    \textit{A Course in Universal Algebra,}
    Springer-Verlag, 1981.


\bibitem{CF2005}
    \textsc{M. Couceiro,}  \textsc{S. Foldes,}
    On closed sets of relational constraints and classes of functions closed under variable substitution,
    \textit{Algebra Universalis} \textbf{54} (2005) 149--165.


\bibitem{CF2007}
    \textsc{M. Couceiro,}  \textsc{S. Foldes,}
    Functional equations, constraints, definability of function classes, and functions of Boolean variables,
    \textit{Acta Cybernet.}\ \textbf{18} (2007) 61--75.


\bibitem{CL-ISMVL}
    \textsc{M. Couceiro,}  \textsc{E. Lehtonen,}
    The arity gap of polynomial functions over bounded distributive lattices,
    arXiv:0910.5131.


\bibitem{CouMar}
    \textsc{M. Couceiro,}  \textsc{J.-L. Marichal,}
    Characterizations of discrete Sugeno integrals as polynomial functions over distributive lattices,
    \textit{Fuzzy Sets and Systems} (2009), doi:10.1016/\linebreak[0]j.fss.\linebreak[0]2009.10.008.


\bibitem{DP}
    \textsc{B. Davey,}  \textsc{H.~A. Priestley,}
    \textit{Introduction to Lattices and Order,} 2nd edition,
    Cambridge University Press, Cambridge, 2002.


\bibitem{DW}
    \textsc{K. Denecke,}  \textsc{S.~L. Wismath,}
    \textit{Universal Algebra and Applications in Theoretical Computer Science,}
    Chapman \& Hall/CRC, Boca Raton, 2002.


\bibitem{DW2009}
    \textsc{K. Denecke,}  \textsc{S.~L. Wismath,}
    \textit{Universal Algebra and Coalgebra,}
    World Scientific, 2009.

    
\bibitem{EH}
    \textsc{B. Eckmann,}  \textsc{P.~J. Hilton,}
    Group-like structures in general categories I---multiplications and comultiplications,
    \textit{Math.\ Ann.}\ \textbf{145} (1962) 227--255.


\bibitem{Goodstein}
    \textsc{R.~L. Goodstein,}
    The solution of equations in a lattice,
    \textit{Proc.\ Roy.\ Soc.\ Edinburgh Sect.\ A} \textbf{67} (1965/1967) 231--242.


\bibitem{GMMP}
    \textsc{M. Grabisch,}  \textsc{J.-L. Marichal,}  \textsc{R. Mesiar,}  \textsc{E. Pap,}
    \textit{Aggregation Functions,}
    Encyclopedia of Mathematics and Its Applications, vol.\ 127,
    Cambridge University Press, Cambridge, 2009.


\bibitem{Gratzer}
    \textsc{G. Grätzer,}
    \textit{Universal Algebra,} 2nd edition,
    Springer-Verlag, 1979.


\bibitem{JKAU}
    \textsc{J. Je\v{z}ek,}  \textsc{T. Kepka,}
    Equational theories of medial groupoids,
    \textit{Algebra Universalis} \textbf{17} (1983) 174--190.


\bibitem{JKRoz}
    \textsc{J. Je\v{z}ek,}  \textsc{T. Kepka,}
    Medial groupoids,
    \textit{Rozpravy \v{C}eskoslovenské Akad.\ V\v{e}d, \v{R}ada Mat.\ P\v{r}írod.\ V\v{e}d} \textbf{93} (1983), 93~pp.


\bibitem{Larose}
    \textsc{B. Larose,}
    On the centralizer of the join operation of a finite lattice,
    \textit{Algebra Universalis} \textbf{34} (1995) 304--313.


\bibitem{MR}
    \textsc{H. Machida,}  \textsc{I.~G. Rosenberg,}
    On the centralizers of monoids in clone theory,
    33rd IEEE International Symposium on Multiple-Valued Logic (ISMVL 2003), 16--19 May 2003, Tokyo, Japan,
    IEEE Computer Society, 2003, pp.~303--308.


\bibitem{MMT}
    \textsc{J.-L. Marichal,}  \textsc{P. Mathonet,}  \textsc{E. Tousset,}
    Characterization of some aggregation functions stable for positive linear transformations,
    \textit{Fuzzy Sets and Systems} \textbf{102} (1999) 293--314.


\bibitem{Pippenger}
    \textsc{N. Pippenger,}
    Galois theory for minors of finite functions,
    \textit{Discrete Math.}\ \textbf{254} (2002) 405--419.


\bibitem{PS}
    \textsc{G. Polák,}  \textsc{\'A. Szendrei,}
    Independent basis for the identities of entropic groupoids,
    \textit{Comment.\ Math.\ Univ.\ Carolin.}\ \textbf{22} (1981) 71--85.


\bibitem{RS}
    \textsc{A. Romanowska,}  \textsc{J.~D. Smith,}
    \textit{Modes,}
    World Scientific, Singapore, 2002.
    
    
\bibitem{Rud01}
    \textsc{S. Rudeanu,}
    \textit{Lattice Functions and Equations,}
    Discrete Mathematics and Theoretical Computer Science Series,
    Springer-Verlag, London, 2001.


\bibitem{ST}
    \textsc{J. Sichler,}  \textsc{V. Trnková,}
    Essential operations in centralizer clones,
    \textit{Algebra Universalis} \textbf{59} (2008) 277--301.


\bibitem{Soublin}
    \textsc{J.-P. Soublin,}
    Étude algébrique de la notion de moyenne,
    \textit{J. Math.\ Pure Appl.}\ \textbf{50} (1971) 53--264.


\bibitem{Stronkowski}
    \textsc{M.~M. Stronkowski,}
    Cancellation in entropic algebras,
    \textit{Algebra Universalis} \textbf{60} (2009) 439--468.


\bibitem{TS}
    \textsc{V. Trnková,}  \textsc{J. Sichler,}
    All clones are centralizer clones,
    \textit{Algebra Universalis} \textbf{61} (2009) 77--95.

\end{thebibliography}
\end{document}